\documentclass[graybox]{svmult}

\usepackage{mathptmx}       
\usepackage{helvet}         
\usepackage{courier}        
\usepackage{type1cm}        
\usepackage{makeidx}         
\usepackage{graphicx}        
\usepackage{multicol}        
\usepackage[bottom]{footmisc}

\usepackage{amsfonts}
\usepackage{amsopn}
\usepackage{amsmath}
\usepackage{amssymb}
\usepackage{hyperref}
\usepackage{thm-restate}
\usepackage{url}

\newcommand{\bolda}{\mathbf{a}}
\newcommand{\boldb}{\mathbf{b}}

\newcommand{\boldp}{\mathbf{p}}
\newcommand{\boldq}{\mathbf{q}}
\newcommand{\boldx}{\mathbf{x}}
\newcommand{\boldy}{\mathbf{y}}
\newcommand{\boldzero}{\mathbf{0}}

\newcommand{\Z}{\mathbb{Z}}
\newcommand{\R}{\mathbb{R}}
\newcommand{\Q}{\mathbb{Q}}
\newcommand{\C}{\mathbb{C}}
\newcommand{\V}{\mathbb{V}}

\newcommand{\aff}{\operatorname{aff}}
\newcommand{\conv}{\operatorname{conv}}
\newcommand{\supp}{\operatorname{supp}}
\newcommand{\vvol}{\operatorname{vol}}
\newcommand{\nvol}{\operatorname{Vol}}
\newcommand{\mvol}{\operatorname{MV}}

\makeindex             

\begin{document}

\title*{A Product Formula for the Normalized Volume of Free Sums of Lattice Polytopes}
\author{Tianran Chen and Robert Davis}
\institute{
    Tianran Chen \at
    Department of Mathematics and Computer Science,
    Auburn University Montgomery,
    Montgomery Alabama.
    \email{ti@nranchen.org}
    \and
    Robert Davis \at
	Department of Mathematics,
	Michigan State University,
	East Lansing, Michigan.
	\email{davisr@math.msu.edu}
}

\maketitle

\abstract{
    The free sum is a basic geometric operation among convex polytopes.
    This note focuses on the relationship between the normalized volume
    of the free sum and that of the summands.
    In particular, we show that the normalized volume of the free sum
    of full dimensional polytopes is precisely the product of
    the normalized volumes of the summands.
}
\abstract*{
    The free sum is a basic geometric operation among convex polytopes.
    This note focuses on the relationship between the normalized volume
    of the free sum and that of the summands.
    In particular, we show that the normalized volume of the free sum
    of full dimensional polytopes is precisely the product of
    the normalized volumes of the summands.
}

\section{Introduction}\label{sec:intro}

A (\emph{convex}) \emph{polytope} $P \subset \R^n$ is the convex hull of
finitely many points in $\R^n$.
Equivalently, $P$ is the bounded intersection of finitely many closed half-spaces in $\R^n$.
We call a polytope \emph{lattice} if its vertices are elements of $\Z^n$.
Polytopes are fascinating combinatorial objects, and lattice polytopes
are especially interesting due to their appearance in many contexts,
such as combinatorics, algebraic statistics, and physics;
see, for example,~\cite{BatyrevDualPolyhedra,sturmfels,ZieglerLectures}

Given polytopes $P \subset \R^m$ and $Q \subset \R^n$, set
\[
    P \oplus Q := \conv\{(P,\boldzero) \cup (\boldzero,Q)\} \subset \R^{m+n}.
\]
If $P$ and $Q$ each contain the origin, then we call $P \oplus Q$ the \emph{free sum}
of $P$ and $Q$.
It will be convenient to use the notation $P',Q'$ for
$(P,\boldzero),(\boldzero,Q) \in \R^{m+n}$, respectively, so that we simply write
$P \oplus Q = \conv \{P' \cup Q'\}$.
In the present contribution, we show that the normalized volume of the free sum
$P \oplus Q$ is precisely the product of the normalized volumes of $P$ and $Q$.

Given a lattice polytope $P \subset \R^n$, its \emph{lattice point enumerator}, defined on the positive integers, is
\[
    L_P(m) := |mP \cap \Z^n|.
\]
Ehrhart \cite{Ehrhart} showed that $L_P(m)$ agrees with a polynomial of degree $\dim P$ with rational coefficients, called the \emph{Ehrhart polynomial}.
Thus, the \emph{Ehrhart series}
\[
    E_P(t) := 1 + \sum_{m \geq 1} L_P(m)t^m
\]
is a rational function of the form
\[
    E_P(t) = \frac{h_P^*(t)}{(1-t)^{\dim P + 1}},
\]
where $h_P^*(t) = \sum_{i = 0}^d h_i^*t^i$ has degree $d \leq \dim P$.
The vector $h^*(P) = (h_0^*,\ldots,h_d^*)$ is called the \emph{$h^*$-vector} of $P$, and Stanley \cite{StanleyDecompositions} proved that it consists of nonnegative integers.
We refer the reader to \cite{BeckRobinsCCD} for background on Ehrhart theory.

The choice of notation $h^*$ evokes parallels with combinatorial $h$-vectors of simplicial complexes.
For example, under certain conditions, the $h^*$-vector of a lattice polytope is the $h$-vector of a simplicial polytope \cite{StanleyDecompositions}.
This is not to be expected in general, though; such a correspondence can fail even in $\R^3$.

\begin{example}\label{ex:reeve}
Let $T_r$ be the tetrahedron
\[
    T_r = \conv\{(0,0,0),(1,0,0),(0,1,0),(1,1,r)\}
\]
where $r$ is a positive integer.
It is not difficult to verify that that $h^*(T_r) = (1,0,r-1)$, but this is not the $h$-vector for a $(d-1)$-dimensional simplicial complex whenever $r > 1$.
\end{example}

Another basic result of Ehrhart theory is that if $P$ is a lattice polytope, then $h_P^*(1)$ is the \emph{normalized volume}
\[
\nvol P := (\dim P)!\vvol P.
\]
where $\vvol P$ is the relative volume of $P$.
Additionally, for all lattice polytopes, $h_0^* = 1$, and there are simple combinatorial interpretations for $h_1^*$ as well as $h_d^*$.
However, the remaining coefficients do not have such elementary descriptions in general.
This has led to a large amount of research being conducted to better understand the behavior of $h^*$-vectors for lattice polytopes, even for quite restricted classes of polytopes.
See \cite{BraunUnimodality} for a recent survey regarding $h^*$-vectors of lattice polytopes.

Define the \emph{(polar) dual} of $P \subset \R^m$ to be
\[
	P^{\vee} = \{ x \in \R^m \mid x^Ty \leq 1 \text{ for all } y \in P\}.
\]
The dual $P^{\vee}$ is again a polytope if and only if $\boldzero$ is an interior point of $P$; otherwise, $P^{\vee}$ is an unbounded polyhedron.
If $P$ and $P^{\vee}$ are both lattice polytopes, then $P$ is called \emph{reflexive}.
A lattice translate of a reflexive polytope is also called reflexive.
Reflexive polytopes were originally introduced by Batyrev \cite{BatyrevDualPolyhedra} to describe mirror symmetry phenomena in supersymmetric string theory.
Since their introduction, many fascinating combinatorial and arithmetic properties of theirs have been discovered \cite{BraunDavisSolus, Conrads, HaaseMelnikov, hibireflexive, Hibi, KreuzerSkarke98, Payne}.
An increasingly well-known result in Ehrhart theory related to reflexive polytopes is the following, also known as Braun's formula.

\begin{theorem}[Braun's formula,\cite{BraunEhrhartFormulaReflexivePolytopes}]
    If $P \subset \R^m$ is a reflexive polytope and $Q \subset \R^n$ is a lattice polytope containing the origin in its interior, then
    \[
        E_{P \oplus Q}(t) = (1-t)E_P(t)E_Q(t).
    \]
    Consequently,
    \[
	    h^*_{P \oplus Q}(t) = h_P^*(t)h_Q^*(t).
    \]
\end{theorem}

The conditions needed for Braun's formula to hold were significantly relaxed in \cite{BJM13}, including a multivariate generalization.
However, we will not need such power here.

Because $h^*_P(1)$ gives the normalized volume of $P$, it follows that if $P \oplus Q$ is a free sum satisfying the conditions of Braun's formula, then
\[
	\nvol_{m+n}(P \oplus Q) = \nvol_m(P)\nvol_n(Q).
\]
If $\boldzero$ is merely on the boundary of $P$ or $Q$, then $P \oplus Q$ is still called a free sum but its $h^*$-polynomial may not factor.
In this note we shall examine instances in which the $h^*$-polynomial may fail to factor
yet the normalized volume of the free sum does factor as the product of
normalized volumes of the summands.

\begin{example}
	Let $P = T_2$ as in Example~\ref{ex:reeve} and let $Q = P \oplus P \subset \R^6$.
	It is easy to verify that $P$ contains no interior lattice point, and so is not reflexive.
	Moreover, a routine computation shows that the $h^*$-vectors of $P$ and $P \oplus P$ are
	$(1,0,1)$ and $(1,0,2,1)$ respectively.
	Therefore the $h^*$-polynomial of $P \oplus P$ is distinct from $(h^*_P(t))^2$.
	However, the normalized volume factoring
	\[
		\nvol_6(P \oplus P) = 4 = \nvol_3(P) \nvol_3(P)
	\]
	remains valid.
\end{example}

The main goal of this article is to establish the general fact that the (top dimensional)
normalized volume of a free sum is always the product of the normalized volumes
of the summands.

\begin{restatable}{theorem}{mainthm}
	\label{thm:mainthm}
	Given full-dimensional convex polytopes $P \subset \R^m$ and $Q \subset \R^n$,
	if both $P$ and $Q$ contain the origin (of $\R^m$ and $\R^n$ respectively),
	then
	\[
		\nvol_{m+n} (P \oplus Q) = \nvol_m (P) \nvol_n(Q).
	\]
\end{restatable}

This theorem demonstrates that even if the distribution of lattice points in the various dilations of $P \oplus Q$ does not carry over from that of $P$ and $Q$, we can still easily recover geometric information about the polytope.

It is easy to see that if $P'$ and $Q'$ do not intersect, then there is no hope of a similar product formula for $\nvol_{m+n}(P \oplus Q)$ to occur in general.
As a simple example, note that if $P = [k,k+1] \subset \R$ for some positive integer $k$ and $Q = [0,1] \subset \R$, then the product of their normalized volumes is $1$ but the normalized volume of $P \oplus Q$ increases as $k \to \infty$.

Section~\ref{sec:notations} recalls important results related to
lattice polytopes and roots of Laurent polynomial systems.
We then present the proof of Theorem~\ref{thm:mainthm} in Section~\ref{sec:proofs}.

We conclude the introduction by remarking that Theorem~\ref{thm:mainthm} is implicit in the work of \cite{StapledonFreeSums}, which considers a \emph{weighted $h^*$-polynomial}.
While Stapledon's results are more general (by considering rational dilations of polytopes), our alternate proof for integer dilations approaches the question via bounds on isolated points of complex varieties.
Through personal communication, we also learned that Tyrrell McAllister produced
(but did not publish) yet another proof of this result via elementary
convex-analytic techniques in 2015~\cite{mcallister_private_2018}.
In comparison, our approach highlights the close connection between
algebra and convex geometry.
Finally, we point out that a similar statement can be made about
co-convex bodies~\cite{alilooee2016generalized}.

\section{Preliminaries} \label{sec:notations}

This section briefly reviews notations and concepts to be used.
Given a polytope $P \subset \R^n$ and a linear form $l(\boldx)$, we call
\[
	l_b(\boldx) := \{\boldx \in \R^n \mid l(\boldx) = b\}
\]
a \emph{supporting hyperplane} of $P$ if either
\[
    l(\boldx) \leq b \text{ for all } \boldx \in P
\]
or
\[
    l(\boldx) \geq b \text{ for all } \boldx \in P.
\]
A \emph{face} of $P$ is a set of the form $l_b(\boldx) \cap P$ for some supporting hyperplane $l_b(\boldx)$ of $P$.
The empty set and $P$ itself are faces of $P$, corresponding to a hyperplane that does not intersect $P$ in the former case, and the ``empty'' hyperplane in the latter.
The \emph{dimension} of a face $F$ is defined as $\dim \aff F$, where $\aff F$ is the affine span of $F$.
By convention, $\dim \emptyset := -1$.
A face of dimension $0$ is called a \emph{vertex}, a face of dimension $1$ is called an \emph{edge}, and a face of dimension $\dim P - 1$ is called a \emph{facet}.

Next, recall that the \emph{Minkowski sum} of two sets $A,B \subset \R^n$,
is
\[
A + B = \{ \bolda + \boldb \in \R^n \mid \bolda \in A, \boldb \in B \}.
\]
Given polytopes $Q_1,\dots,Q_n \subset \R^n$, and positive
scalars $\lambda_1,\dots,\lambda_n$, the Minkowski sum
$\lambda_1 Q_1 + \cdots + \lambda_n Q_n$ is also a polytope, and its volume is a homogeneous polynomial
in $\lambda_1,\dots,\lambda_n$.
In it, the coefficient of $\lambda_1 \cdots \lambda_n$ is the \emph{mixed volume}\footnote{%
    An alternative definition for mixed volume is the coefficient of
    $\lambda_1 \cdots \lambda_n$ in the above polynomial divided by $n!$.
}
\cite{minkowski_theorie_1911}
of $Q_1,\dots,Q_n$, denoted $\mvol(Q_1,\dots,Q_n)$.
It agrees with normalized volume ``on the diagonal,'' that is,
$\mvol(Q,\dots,Q) = \nvol_n(Q)$.

Though the main results to be established in this note concern a geometric
property of polytopes, our proofs take a decidedly algebraic approach
via the theory for counting roots of Laurent polynomial systems.
A \emph{Laurent polynomial} in $x_1,\dots,x_n$ is an expression of the form
\[
	p = \sum_{k=1}^m c_k x_1^{a_{k,1}} \cdots x_n^{a_{k,n}}
\]
where $c_k \in \C^* = \C \setminus \{0\}$ and $a_{k,j} \in \Z$.
Thus, the Laurent polynomials are exactly the expressions obtained when localizing the usual polynomial ring $\C[x_1,\dots,x_n]$ at each indeterminate.
We adopt the compact notation $p = \sum_{\bolda \in S} c_{\bolda} \boldx^{\bolda}$
where each $\bolda = (a_1,\dots,a_n) \in \Z^n$ encodes the exponents of a term,
$\boldx^{\bolda} := x_1^{a_1} \cdots x_n^{a_n}$, and the set $S \subset \Z^n$
collecting all such exponents is known as the \emph{support} of $p$,
denoted $\supp(p)$.

For a Laurent polynomial system $P = (p_1,\dots,p_m)$ in $\boldx = (x_1,\dots,x_n)$,
we define $\V^*(P) := \{ \boldx \in (\C^*)^n \mid P(\boldx) = \boldzero \}$,
which is the common zero set of $P$ within $(\C^*)^n$.
Here, the dimension $n$ (of the ambient space) is generally clear from context
and hence does not directly appear in the notation $\V^*(P)$.
Such a zero set has rich internal structures
(e.g. the structure of a quasi-projective algebraic set
\cite{hartshorne_algebraic_1977}).
In particular, $\V^*(P)$ consists of finitely many irreducible components
each with a well-defined dimension.
Of special interest to us are the $0$-dimensional components of $\V^*(P)$, which are the \emph{isolated points}:
elements $\boldx$ for which there is an open set that contains $\boldx$ but does not contain any other points in $\V^*(P)$.
This subset will be denoted by $\V_0^*(P)$.
Our proof makes frequent use of the following two theorems:

\begin{theorem}[Kushnirenko's Theorem \cite{kushnirenko_newton_1976}]
\label{thm:kushnirenko}
    Given a Laurent polynomial system $P = (p_1,\dots,p_n)$ in
    $\boldx = (x_1,\dots,x_n)$,
    if the supports of $p_1,\dots,p_n$ are identical,
    that is, if $S:= \supp(p_1) = \cdots = \supp(p_n)$, then
\[
    	|\V_0^*(P)| \le n! \vvol_n (\conv(S)).
\]
\end{theorem}


\begin{theorem}[Bernshtein's First Theorem \cite{bernshtein_number_1975}]
\label{thm:bernshtein-a}
    For a Laurent polynomial system $P(x_1,\dots,x_n) = (p_1,\dots,p_n)$,
\[
	|\V_0^*(P)| \le \mvol (\conv(\supp(p_1)),\dots,\conv(\supp(p_n))).
\]
\end{theorem}

\begin{remark}\label{rmk:exactness}
    Implicit in the above theorems is the fact that the upper bounds for $|\V_0^*(P)|$
    given in both statements are ``generically exact'' in the sense that if the coefficients
    are chosen at random, the probability of picking coefficients for which the bounds
    are not exact is zero.
    Stated more precisely, the set of coefficients for which the bounds are exact
    form a nonempty Zariski open set among the set of all possible choices of coefficients.
    The generically exact solution bound given by Theorem~\ref{thm:bernshtein-a} has since
    been known as the BKK bound after the circle of works by
    Bernshtein~\cite{bernshtein_number_1975},
    Kushnirenko~\cite{kushnirenko_newton_1975,kushnirenko_newton_1976},
    and Khovanskii~\cite{khovanskii_newton_1978}.
\end{remark}

Our main proof additionally relies on the following result on the connection between
mixed volume and normalized volume established independently in~\cite{Bihan2017}
and~\cite{chen_unmixing_2017}.

\begin{theorem}[{\cite[Theorem 2]{chen_unmixing_2017}}]\label{thm:mainthma}
    Given nonempty finite sets $S_1,\dots,S_n \subset \Q^n$,
    let $\tilde{S} = S_1 \cup \cdots \cup S_n$.
    If every positive-dimensional face $F$ of $\conv(\tilde{S})$ satisfies
    one of the following conditions:
    \begin{description}[leftmargin=5ex]
    	\item[(A)]
    	$F \cap S_i \ne \varnothing$ for all $i \in \{1,\dots,n\}$;

    	\item[(B)]
    	$F \cap S_i$ is a singleton for some $i \in \{1,\dots,n\}$;

    	\item[(C)]
    	For each $i \in I := \{ i \mid F \cap S_i \ne \varnothing \}$,
    	$F \cap S_i$ is contained in a common coordinate subspace of
    	dimension $|I|$, and the projection of $F$ in this subspace is
    	of dimension less than $|I|$;
    \end{description}
    then $\mvol (\conv(S_1), \dots, \conv(S_n)) = \nvol_n (\conv(\tilde{S}))$.
\end{theorem}

Note that if $\conv(\tilde{S})$ is not full-dimensional,
then both sides of the above equations will be zero.

\section{Proofs of the main results}\label{sec:proofs}

We now turn to the main result, Theorem \ref{thm:mainthm},
which we shall recall here:

\mainthm*

\begin{proof}
    First, we consider the cases where $P$ and $Q$ are both lattice polytopes.
	Let $S \subset \Z^m$ and $T \subset \Z^n$ be the set of vertices of
	$P$ and $Q$ respectively.
	Also let $P' = \{ (\boldp,\boldzero) \in \R^{m+n} \mid \boldp \in P \}$
	and      $Q' = \{ (\boldzero,\boldq) \in \R^{m+n} \mid \boldq \in Q \}$.
	The finite sets $S',T' \subset \Z^{m+n}$ are the sets of vertices of
	$P'$ and $Q'$ respectively.
	Consider the two sets of Laurent monomials
	$\boldx^S = \{ \boldx^{\bolda} \mid \bolda \in S \}$ and
	$\boldy^T = \{ \boldy^{\boldb} \mid \boldb \in T \}$ in the variables
	$\boldx = (x_1,\dots,x_m)$ and $\boldy = (y_1,\dots,y_n)$.
	We take the linear combinations
	\begin{align*}
		f_i(\boldx) &:= \sum_{\bolda \in S} c_{i,\bolda} \boldx^{\bolda}
		\quad \text{for } i=1,\dots,m &
		g_j(\boldy) &:= \sum_{\boldb \in T} c'_{j,\boldb} \boldy^{\boldb}
		\quad \text{for } j=1,\dots,n
	\end{align*}
	where the coefficients $c_{i,\bolda}$ and $c'_{j,\boldb}$ can be taken generically,
    which is possible by the discussion in Remark~\ref{rmk:exactness}.
	We can now form a system of $m+n$ Laurent polynomial equations
	\[
		H(\boldx,\boldy) := (
			f_1(\boldx),\, \ldots \,,\, f_m(\boldx), \;
			g_1(\boldy),\, \ldots \,,\, g_n(\boldx)
		)
	\]
	in the $m+n$ variables $(\boldx,\boldy) = (x_1,\dots,x_n,y_1,\dots,y_n)$.
	It is easy to verify that
	\[
		\supp(H) = (\underbrace{S',\dots,S'}_m \,,\, \underbrace{T',\dots,T'}_n)
	\]
	Since the $c_{i,\bolda}$ and $c'_{j,\boldb}$ are generic, the BKK bound for $H$
	given by Theorem~\ref{thm:bernshtein-a} is exact.
	In that case,
	\[
		| \V^*_0 (H) | =
		\mvol(\conv(S'),\dots,\conv(S'),\conv(T'),\dots,\conv(T')).
	\]
	We shall now show, using Theorem~\ref{thm:mainthma}, that this quantity is also the
	normalized volume of
	\[
	    \conv(S' \cup T') = \conv(P' \cup Q') = P \oplus Q.
	\]

	Let $F$ be a positive dimensional proper face of $P \oplus Q$.
	We shall show $F$ always satisfy one of the conditions listed in
    Theorem~\ref{thm:mainthma}.

	\noindent\textbf{(Case I)}
	First, suppose $F$ contains $P'$ or $Q'$.
	Then by assumption, $F$ intersects both $P'$ and $Q'$ since the two must intersect
	at the origin according to the definition of a free sum.
	Therefore $F$ satisfies condition (A) in Theorem~\ref{thm:mainthma}.
	\smallskip

	\noindent\textbf{(Case II)}
	If $F$ is a proper face of $P'$ and $F \cap Q' = \varnothing$,
	then
	\[ F \subset P' \subset \{ (\boldp,\boldzero) \in \R^{m+n} \mid \boldp \in \R^m \}.\]
	In other words, $F$ is contained in a coordinate subspace of dimension $m$
	and its projection in this subspace is of dimension less than $m$.
	Therefore $F$ satisfies condition (C) in Theorem~\ref{thm:mainthma}.
	By the same argument, $F$ would satisfy the same condition if $F$ is a proper face
	of $Q'$ instead and $F \cap P' = \varnothing$.
	\smallskip

	\noindent\textbf{(Case III)}
	Finally, suppose $F$ does not contain $P'$ or $Q'$
	nor is it contained in $P'$ or $Q'$.
	Since a face of $P \oplus Q$ must contain vertices of $P \oplus Q$ which are necessarily
	of the form $(\boldp,\boldzero)$ or $(\boldzero,\boldq)$ for $\boldp \in P$ or $\boldq \in Q$,
	$F$ must intersect both $P'$ and $Q'$.
	That is, it satisfies condition (A) in Theorem~\ref{thm:mainthma}.
	\smallskip

	The above cases exhausts all the possibilities, therefore by Theorem \ref{thm:mainthma},
	\begin{align*}
		\mvol(\conv(S'),\dots,\conv(S'),\conv(T'),\dots,\conv(T'))
		&= \nvol(\conv(S' \cup T')) \\
		&= \nvol(\conv(P' \cup Q')) \\
		&= \nvol(P \oplus Q).
	\end{align*}

	Now focusing on $|\V^*_0(H)|$, we can see that the first $m$ equations in
	$H = \boldzero$ only involve variables $\boldx = (x_1,\dots,x_m)$
	and the remaining equations only involve variables $\boldy = (y_1,\dots,y_n)$.
	So the solutions to $H (\boldx,\boldy) = \boldzero$ are precisely the points
	of the form $(\boldx,\boldy)$ such that $F(\boldx) = \boldzero$
	and $G(\boldy) = \boldzero$, where
	\[
		F(\boldx) = (f_1(\boldx),\dots,f_m(\boldx)) \text{ and } G(\boldy) = (g_1(\boldy),\dots,g_n(\boldy)).
	\]
	Consequently,
	\[
		| \V^*_0 (H) | =
		| \V^*_0(F) | \cdot | \V^*_0(G) | =
		\nvol_m(P) \cdot \nvol_n(Q)
	\]
	by the Kushnirenko's Theorem (Theorem~\ref{thm:kushnirenko}).
	Therefore we have a chain of equalities
	\[
		\nvol(P \oplus Q) =
		\mvol(P',\dots,P',Q',\dots,Q') =
		|\V^*_0(H)| =
		\nvol_m(P) \cdot \nvol_n(Q).
	\]

	The above shows the statement holds for two lattice polytopes $P$ and $Q$.
	Since the volume forms $\vvol_{m+n}$, $\vvol_m$, and $\vvol_n$
	are homogeneous of degree $m+n$, $m$, and $n$ respectively
	under uniform scaling,
	this result directly extends to polytopes with vertices in
	$\Q^m$ and $\Q^n$ respectively.
	Then by the continuity of volume forms, the statement must also hold for
	polytopes with vertices of real coordinates.
\end{proof}


\begin{thebibliography}{10}

\bibitem{alilooee2016generalized}
{\sc A.~Alilooee, I.~Soprunov, and J.~Validashti}, {\em Generalized
  multiplicities of edge ideals}, Journal of Algebraic Combinatorics,  (2016),
  pp.~1--32.

\bibitem{BatyrevDualPolyhedra}
{\sc V.~V. Batyrev}, {\em Dual polyhedra and mirror symmetry for {C}alabi-{Y}au
  hypersurfaces in toric varieties}, J. Algebraic Geom., 3 (1994),
  pp.~493--535.

\bibitem{BJM13}
{\sc M.~Beck, P.~Jayawant, and T.~B. McAllister}, {\em Lattice-point generating
  functions for free sums of convex sets}, Journal of Combinatorial Theory,
  Series A, 120 (2013), pp.~1246 -- 1262,
  \href{http://dx.doi.org/http://dx.doi.org/10.1016/j.jcta.2013.03.007}{doi:\nolinkurl{http://dx.doi.org/10.1016/j.jcta.2013.03.007}},
  \url{http://www.sciencedirect.com/science/article/pii/S0097316513000599}.

\bibitem{BeckRobinsCCD}
{\sc M.~Beck and S.~Robins}, {\em Computing the continuous discretely},
  Undergraduate Texts in Mathematics, Springer, New York, 2007.

\bibitem{bernshtein_number_1975}
{\sc D.~N. Bernshtein}, {\em {The number of roots of a system of equations}},
  Functional Analysis and its Applications, 9 (1975), pp.~183--185.

\bibitem{Bihan2017}
{\sc F.~Bihan and I.~Soprunov}, {\em {Criteria for strict monotonicity of the
  mixed volume of convex polytopes}}, arXiv:1702.07676 [math],  (2017),
  \url{http://arxiv.org/abs/1702.07676 http://www.arxiv.org/pdf/1702.07676.pdf
  https://arxiv.org/abs/1702.07676}.

\bibitem{BraunEhrhartFormulaReflexivePolytopes}
{\sc B.~Braun}, {\em An {E}hrhart series formula for reflexive polytopes},
  Electron. J. Combin., 13 (2006), pp.~Note 15, 5 pp. (electronic).

\bibitem{BraunUnimodality}
{\sc B.~Braun}, {\em Unimodality problems in Ehrhart theory}, Springer
  International Publishing, Cham, 2016, pp.~687--711,
  \href{http://dx.doi.org/10.1007/978-3-319-24298-9_27}{doi:\nolinkurl{10.1007/978-3-319-24298-9_27}},
  \url{http://dx.doi.org/10.1007/978-3-319-24298-9_27}.

\bibitem{BraunDavisSolus}
{\sc B.~Braun, R.~Davis, and L.~Solus}, {\em Detecting the integer
  decomposition property and ehrhart unimodality in reflexive simplices},
  (2016).
\newblock http://arxiv.org/abs/1608.01614.

\bibitem{chen_unmixing_2017}
{\sc T.~Chen}, {\em {Unmixing the mixed volume computation}}, arXiv:1703.01684
  [math],  (2017), \url{http://arxiv.org/abs/1703.01684}.

\bibitem{Conrads}
{\sc H.~Conrads}, {\em Weighted projective spaces and reflexive simplices},
  Manuscripta Math., 107 (2002), pp.~215--227,
  \href{http://dx.doi.org/10.1007/s002290100235}{doi:\nolinkurl{10.1007/s002290100235}},
  \url{http://dx.doi.org/10.1007/s002290100235}.

\bibitem{Ehrhart}
{\sc E.~Ehrhart}, {\em Sur les poly\`edres rationnels homoth\'etiques \`a
  {$n$}\ dimensions}, C. R. Acad. Sci. Paris, 254 (1962), pp.~616--618.

\bibitem{HaaseMelnikov}
{\sc C.~Haase and I.~V. Melnikov}, {\em The reflexive dimension of a lattice
  polytope}, Ann. Comb., 10 (2006), pp.~211--217.

\bibitem{hartshorne_algebraic_1977}
{\sc R.~Hartshorne}, {\em {Algebraic geometry}}, no.~52, Springer, 1977.

\bibitem{hibireflexive}
{\sc T.~Hibi}, {\em Ehrhart polynomials of convex polytopes, $h$-vectors of
  simplicial complexes, and nonsingular projective toric varieties}, Discrete
  and Computational Geometry: Papers from the {DIMACS} Special Year, 6 (1991),
  p.~165–177.

\bibitem{Hibi}
{\sc T.~Hibi}, {\em Algebraic Combinatorics on Convex Polytopes}, Carslaw
  Publications, Australia, 1992.

\bibitem{khovanskii_newton_1978}
{\sc A.~G. Khovanskii}, {\em {Newton polyhedra and the genus of complete
  intersections}}, Functional Analysis and Its Applications, 12 (1978),
  pp.~38--46,
  \href{http://dx.doi.org/10.1007/BF01077562}{doi:\nolinkurl{10.1007/BF01077562}},
  \url{http://dx.doi.org/10.1007/BF01077562}.

\bibitem{KreuzerSkarke98}
{\sc M.~Kreuzer and H.~Skarke}, {\em Classification of reflexive polyhedra in
  three dimensions}, Adv. Theor. Math. Phys., 2 (1998), pp.~853--871.

\bibitem{kushnirenko_newton_1975}
{\sc A.~G. Kushnirenko}, {\em {A Newton polyhedron and the number of solutions
  of a system of k equations in k unknowns}}, Usp. Math. Nauk, 30 (1975),
  pp.~266--267.

\bibitem{kushnirenko_newton_1976}
{\sc A.~G. Kushnirenko}, {\em {Newton polytopes and the Bezout theorem}},
  Functional Analysis and Its Applications, 10 (1976), pp.~233--235,
  \href{http://dx.doi.org/10.1007/BF01075534}{doi:\nolinkurl{10.1007/BF01075534}},
  \url{http://link.springer.com/article/10.1007/BF01075534}.

\bibitem{mcallister_private_2018}
{\sc T.~McAllister}.
\newblock {Private Communication}, 2018.

\bibitem{minkowski_theorie_1911}
{\sc H.~Minkowski}, {\em {Theorie der konvexen Korper, insbesondere Begrundung
  ihres Oberflachenbegriffs}}, Gesammelte Abhandlungen von Hermann Minkowski, 2
  (1911), pp.~131--229.

\bibitem{Payne}
{\sc S.~Payne}, {\em Ehrhart series and lattice triangulations}, Discrete
  Comput. Geom., 40 (2008), pp.~365--376.

\bibitem{StanleyDecompositions}
{\sc R.~P. Stanley}, {\em Decompositions of rational convex polytopes}, Ann.
  Discrete Math., 6 (1980), pp.~333--342.
\newblock Combinatorial mathematics, optimal designs and their applications
  (Proc. Sympos. Combin. Math. and Optimal Design, Colorado State Univ., Fort
  Collins, Colo., 1978).

\bibitem{StapledonFreeSums}
{\sc A.~Stapledon}, {\em Counting lattice points in free sums of polytopes},
  Journal of Combinatorial Theory, Series A,  (2017).
\newblock To appear.

\bibitem{sturmfels}
{\sc B.~Sturmfels}, {\em Gr\"obner {B}ases and {C}onvex {P}olytopes}, vol.~8 of
  University Lecture Series, American Mathematical Society, Providence, RI,
  1996.

\bibitem{ZieglerLectures}
{\sc G.~M. Ziegler}, {\em Lectures on polytopes}, vol.~152 of Graduate Texts in
  Mathematics, Springer-Verlag, New York, 1995.

\end{thebibliography}
\end{document}